\documentclass[12pt]{amsart}

\usepackage[all]{xy}
\usepackage{graphics}
\usepackage{color}
\usepackage[T1]{fontenc}
\usepackage{parskip}
\usepackage[colorlinks]{hyperref}
\usepackage{enumitem}
\usepackage{mathtools}

\usepackage{wrapfig}
\usepackage{tikz}
\usetikzlibrary{decorations.markings}

\usepackage[colorinlistoftodos,prependcaption]{todonotes}

\newcommand{\B}[1]{{\mathbf #1}}

\newtheorem{theorem}[subsection]{Theorem}
\newtheorem{corollary}[subsection]{Corollary}
\newtheorem{lemma}[subsection]{Lemma}

\theoremstyle{definition}

\newtheorem{example}[subsection]{Example}

\theoremstyle{remark}
\newtheorem{remark}[subsection]{Remark}

\newtheorem{conjecture}[subsection]{Conjecture}
\numberwithin{figure}{section}
\numberwithin{table}{section}
\numberwithin{equation}{section}

\newcommand{\OP}{\operatorname}

\begin{document}

\title{Qualitative counting closed geodesics}
\author[Karlhofer]{Bastien Karlhofer}
\address{BK: University of Aberdeen}
\email{r01bdk17@abdn.ac.uk}
\author[K\k{e}dra]{Jarek K\k{e}dra}
\address{JK: University of Aberdeen and University of Szczecin}
\email{kedra@abdn.ac.uk}
\author[Marcinkowski]{Micha\l\ Marcinkowski}
\address{MM: IMPAN, Wroc\l aw}
\email{marcinkow@math.uni.wroc.pl}
\author[Trost]{Alexander Trost}
\address{AT: University of Aberdeen}
\email{r01aat17@abdn.ac.uk}

\begin{abstract}

We investigate the geometry of word metrics on fundamental groups of manifolds
associated with the generating sets consisting of elements represented by
closed geodesics. We ask whether the diameter of such a metric is finite or
infinite. The first answer we interpret as an abundance of closed geodesics,
while the second one as their scarcity. We discuss examples for both cases.

\end{abstract}

\maketitle
\section{Introduction}
It is a classical observation due to John Milnor \cite{MR0232311} and Albert
Schwarz \cite{MR0075634} that the word metric on the fundamental group of a
closed manifold carries information about the Riemannian metric of the
universal cover (the metrics are quasi-isometric).  In this approach the word
metric on the fundamental group is associated with a finite generating set.
In the present paper we explore the word metrics on the fundamental group
associated with geometrically meaningful generating sets. Specifically, we
consider generating sets consisting of closed local geodesics.  We then ask the
most basic question as to whether the diameter of such a word metric is finite
or infinite. The first answer is interpreted as abundance of closed geodesics
while the second as their scarcity.  We present examples for both cases.

\subsection{Statement of the results}
Let $(M,d)$ be a complete Riemannian manifold and let $C_x$ denote the set of
closed local geodesics based at $x\in M$.  Let $\Gamma_x \subseteq \pi_1(M,x)$
denote the subgroup of the fundamental group generated by elements represented
by closed local geodesics. We are interested in the word norm on $\Gamma_x$
associated with the set $S$ of the elements represented by closed geodesics.
We call it the {\em closed geodesic norm}. We apply methods of geometric group
theory to prove the following results.

\begin{theorem}\label{T:negative}
Let $(M,d)$ be a closed Riemannian manifold of negative curvature admitting a
geodesic symmetry through $x\in M$.  If $\Gamma_x$ is nonabelian then the
diameter of the closed geodesic norm is infinite.
\end{theorem}

The situation changes if the manifold is only non-positively curved.
A rich source of examples is provided by locally symmetric spaces
$M = \Gamma \backslash G/K$, where $G$ is a semisimple Lie group,
$K\subset G$ a maximal compact subgroup and $\Gamma \subset G$ a lattice.
The natural metric on $M$ is non-positively curved and we have the
following result. 

\begin{theorem}\label{T:chevalley}
Let $(M,d)$ be a complete Riemannian manifold of nonpositive curvature
admitting a geodesic symmetry through $x\in M$.  If $\pi_1(M,x)$ is isomorphic
to a finite index subgroup in an irreducible $S$-arithmetic Chevalley group of
rank at least $2$ then the diameter of the closed geodesic norm on $\Gamma_x$
is finite.
\end{theorem}

Our proof of Theorem \ref{T:chevalley} amounts to showing that the closed
geodesic norm is bounded above by a conjugation invariant norm and then we use
the fact that such norms have finite diameter on S-arithmetic Chevalley groups
\cite{MR2819193,MR2819193-add}.  It would be interesting for find a direct
geometric argument which would prove a more general statement. 

\begin{conjecture}\label{conjecture}
Let $M=\Gamma \backslash G/K$ be a locally symmetric space of rank
at least $2$. If the lattice $\Gamma$ is invariant under the Cartan
involution then the diameter of the closed geodesic norm is finite.
\end{conjecture}

An equivalent form of the above conjecture is that the diameter of the word
norm on $\Gamma$ associated with the generating set consisting of elements
invariant under the Cartan involution is finite.  More generally, it is not
known whether conjugation invariant norms on lattices in Lie groups of rank at
least $2$ have finite diameter.  A piece of evidence that their diameter may be
finite comes from the fact that such lattices do not admit unbounded
quasimorphisms. The above Conjecture \ref{conjecture} is a much weaker
statement in this direction.

\subsection{A comment on counting closed geodesics} 

Classically, counting closed geodesics is done in the form of estimates of the
number of geodesics of a given length \cite{MR630586}.  Here, we propose a
different way of counting. Namely, by measuring how big the subgroup of the
fundamental group generated by closed local geodesics is and whether it has
finite or infinite diameter with respect to the closed geodesic norm.
Finite diameter of the closed geodesic norm is interpreted as
abundance of closed local geodesics and infinite diameter as their scarcity.
For example, on a flat torus every element of the fundamental group is
represented by a closed local geodesic so $\Gamma_x=\pi_1(\B T^n,x)$ and the
norm has diameter one. On the other hand, on a hyperbolic punctured torus, we
have that $\Gamma_x=\pi_1(\B T^2\setminus\{0\})=\B F_2$ for a suitably chosen
basepoint and the closed geodesic norm is equivalent to the palindromic length
on the free group; see Example \ref{E:torus} for details.

\subsection{Examples}
\begin{example}[Hyperbolic punctured torus]\label{E:torus}

Let $(M,d)$ be the hyperbolic punctured torus viewed as a quotient of an ideal
hyperbolic square. Let the basepoint be represented by the centre of the square.
Then the generators of the fundamental group are represented by closed geodesics
(drawn in blue in the figure below).
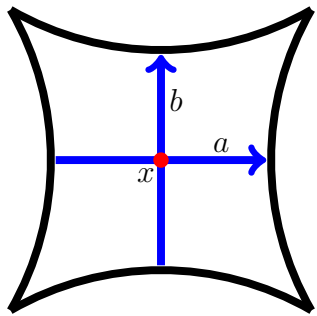
\begin{figure}[h]\label{F:}
\begin{tikzpicture}[line width=3pt, scale=0.2]
\draw (10,0) arc (60:120:20);
\draw (10,20) arc (-60:-120:20);
\draw (-10,0) arc (-30:30:20);
\draw (10,20) arc (150:210:20);
\draw[blue,->] (-7,10) -- (7,10);
\draw[blue,->] (0,3) -- (0,17);
\filldraw[red]
(0,10) circle (7pt);
\draw (-1,9) node {$x$};
\draw (1,14) node {$b$};
\draw (4,11) node {$a$};
\end{tikzpicture}
\caption{Punctured torus. }
\end{figure}
The central symmetry of the square defines a geodesic symmetry $I\colon M\to M$
such that it acts on the fundamental group $\pi_1(M,x) = \B F_2 = \langle a,b \rangle$
by inverting generators. It follows that closed local geodesics represent palindromes
in the free group $\B F_2$. Indeed, $I(w(a,b)) = w(a^{-1},b^{-1})$ is equal to
$w(a,b)^{-1}$ if and only if the reduced word $w(a,b)$ is a palindrome.
Thus the closed geodesic norm is equal to the {\em palindromic length} on
$\B F_2$ and this is known to have infinite diameter \cite{MR2125453},
\cite[Example 6.7]{MR3730755}.
\hfill $\diamondsuit$
\end{example}

\begin{example}[Hyperbolic closed surface I]\label{E:closed-surface}

Let $\Sigma$ be a closed hyperbolic surface of genus $g$ obtained as a quotient
of a regular hyperbolic $4g$-gon in which the opposite sides are identified and
with its centre representing the basepoint. As in the case of the punctured
torus the central symmetry defines a geodesic symmetry which is the
hyperelliptic involution.  Also in this case $\Gamma_x=\pi_1(\Sigma,x)$. It
follows from Theorem~\ref{T:negative} that the diameter of the closed geodesic
norm is infinite.  \hfill $\diamondsuit$

\end{example}

\begin{example}[Hyperbolic closed surface II]\label{E:genus=2}
Let $\Sigma$ be a closed hyperbolic surface of genus $2$ obtained as
a quotient of a regular hyperbolic octagon with identifications
which yield the following presentation of the fundamental group
$$
\pi_1(\Sigma,x) = 
\left\langle a,b,c,d \ |\ [a,b][c,d]=1 \right\rangle.
$$
As before the basepoint is represented by the centre of the octagon
and its central symmetry descends to a geodesic symmetry $I$ of 
$\Sigma$. Observe, that the homomorphism induced by $I$ on
the first homology is defined by $I_*[a] = [c]$ and $I_*[b]=[d]$.
Thus the subspace of $H_1(\Sigma;\B R)$ consisting of
elements such that $I_*(z)=-z$ is $2$-dimensional
generated by $[a]-[c]$ and $[b]-[d]$. In particular, the
subgroup $\Gamma_x\subseteq \pi_1(\Sigma,x)$ is infinite and of infinite index.
It is not difficult to see that is  also nonabelian.
Thus it follows from Theorem \ref{T:negative} that
the closed geodesic norm on $\Gamma_x$ has infinite diameter.
\hfill $\diamondsuit$
\end{example}

\begin{example}[Closed hyperbolic $3$-manifold]\label{E:hyperbolic-3d}

Let $D\subseteq \B H^3$ be a right-angled regular hyperbolic dodecahedron and
let $W\subset \OP{Iso}(\B H^3)$ be the right-angled Coxeter group of isometries
of the hyperbolic space generated by the reflections in the faces of $D$. Let
$\Pi$ be the kernel of the homomorphism $W\to (\B Z/2\B Z)^6$ which sends
reflections through the opposite faces to the same generator.  
Then $M= \B H^3/\Pi$ is a closed hyperbolic manifold glued from $2^6$
dodecahedra. Thus $\Pi=\pi_1(M,x)$. Moreover, the geodesic symmetry 
$I\colon \B H^3\to \B H^3$ at the centre of the dodecahedron $x\in D$
descends to a geodesic symmetry of $M$.  To see this observe that if $s\in W$
is a generator then $IsI=s'$, where $s'\in W$ is the reflection in the face of
$D$ opposite to the face of reflection $s$. This means that the conjugation by
$I$ preserves $W$ and, moreover, $Iss'I = s's = (ss')^{-1}$. Since conjugates
of the six elements $ss'$ by elements of $W$ generate $\Pi$ we obtain that the
conjugation by $I$ preserves $\Pi$ and that $\Gamma_x = \pi_1(M,x)$.
It follows from Theorem \ref{T:negative} that the closed geodesic norm
in $\pi_1(M,x)$ has infinite diameter.
\hfill $\diamondsuit$
\end{example}

\begin{example}\label{E:reid}
Chinburg and Reid \cite{MR1243786} proved that there are infinitely many
noncomensurable examples of closed hyperbolic $3$-manifolds in which all
closed geodesics are simple. Let $M$ be such a manifold. 
It follows that $M$ cannot admit a geodesic symmetry through a point
$x\in M$ contained in a closed geodesic. For if $I\colon M\to M$
was a geodesic symmetry through a point $x\in \gamma$, where $\gamma$
is a geodesic segment with endpoints at $x$ then $\gamma*I(\gamma)$ would
be a closed geodesic with a self-intersection at $x$. 

Furthermore, Jones and Reid \cite{MR1458971} proved later that if two closed
geodesics in $M$ intersect then they are perpendicular. They moreover, proved
that $M$ has points at which at least two closed geodesic intersect. Let $x\in
M$ be such a point. It follows that at most three closed geodesics can
intersect at $x$ for dimensional reasons and hence the group $\Gamma_x$
is finitely generated.
\hfill $\diamondsuit$
\end{example}

\begin{example}[Locally symmetric space of higher rank]\label{E:locally-symmetric}
Let $G$ be a non-compact semisimple Lie group, $K\subset G$ its maximal
compact subgroup and $\Gamma\subset G$ a lattice. If $I\colon G\to G$
is a Cartan involution preserving the lattice (setwise) then
it descends to a geodesic symmetry of the locally symmetric
space $\Gamma\backslash G/K$. 

Let $\Gamma\subseteq \OP{SL}(n,\B Z)$ be a finite index subgroup
so that the locally symmetric space 
$M=\Gamma\backslash \OP{SL}(n,\B R)/\OP{SO}(n)$
is a manifold. The geodesic symmetry is given by the inverse-transpose
and hence the closed geodesics represent symmetric matrices
of $\Gamma$. If $n=2$ then the space is hyperbolic and
the closed geodesic norm has infinite diameter. If $n>2$ then
$\OP{SL}(n,\B Z)$ is an arithmetic Chevalley group of rank
at least $2$ and it follows from Theorem \ref{T:chevalley} that
the diameter of the closed geodesic norm is finite. Observe
that in this case the group $\Gamma_x$ is infinite.
\hfill $\diamondsuit$
\end{example}

\subsection*{Acknowledgements} This work was partly funded by the Leverhulme
Trust Research Project Grant RPG-2017-159. MM is supported by the grant
Sonatina 2018/28/C/ST1/00542 funded by Narodowe Centrum Nauki.
MM and JK were partially supported by SFB 1085 ``Higher Invariants''
funded by Deutsche Forschungsgemeinschaft.

\section{Definitions and supporting results}\label{S:definitions}

\subsection{Geodesics}
We use terminology from \cite{MR1744486}. Let $(M,d)$ be a metric space and let
$x,y\in M$.  A map $\gamma\colon [a,b]\to M$ is called a {\em geodesic from $x$
to $y$} if $d(\gamma(s),\gamma(t))=|s-t|$ for every $s,t\in [a,b]$ and
$\gamma(a)=x$ and $\gamma(b)=y$. The image of such $\gamma$ is called a {\em
geodesic segment}. A {\em local geodesic} is a map $\gamma\colon [a,b]\to M$
such that for every $c\in [a,b]$ there exists an $\epsilon>0$ such
that $d(\gamma(s),\gamma(t))=|s-t|$ for every $s,t\in [c-\epsilon,c+\epsilon]$.
A metric space $(M,d)$ is called a {\em geodesic metric space} if every two
points of $M$ can be joined by a geodesic. A complete connected Riemannian manifold
is a geodesic metric space.

Let $(\B S^1,g)$ denote a circle with the standard metric of total length $2\pi$.
A (locally) isometric embedding $c\colon \B S^1\to M$ is called a {\em
closed (local) geodesic}. 
If $\gamma\colon [a,b]\to M$ is a path then its {\em reverse}
$\overline{\gamma}\colon [a,b]\to M$ is defined by
$\overline{\gamma}(t)=\gamma(a+b-t)$. We define similarly the reverse of a loop
$\gamma\colon \B S^1\to M$.
A {\em geodesic symmetry} $I\colon M\to M$ at $x\in M$ is an isometry such
that $(I\circ \gamma)(t) = \gamma(-t)=\overline{\gamma}(t)$ 
for every geodesic $\gamma\colon [-a,a]\to M$ such that $\gamma(0)=x$.

\begin{lemma}\cite[Theorem 4.13, Chapter II.4]{MR1744486}\label{L:geodesic-rep}
Let $(M,d)$ be a complete, non-positively curved Riemannian manifold with
a basepoint $x\in M$. Then every element $g\in \pi_1(M,x)$ is
represented by a unique local geodesic $\gamma\colon [0,a]\to M$.
\end{lemma}

\begin{lemma}\cite[Theorem 3.8.14]{MR1330918}\label{L:closed-geodesic}
If $(M,d)$ is closed Riemannian manifold of negative curvature then
every free homotopy class of loops is represented by a unique closed
local geodesic. In particular, each conjugacy class in $\pi_1(M,x)$ is
represented by a unique closed local geodesic.
\end{lemma}

\subsection{Quasimorphisms and norms on groups}
Let $G$ be a group. A function $\psi\colon G\to \B R$ is called
a {\em quasimorphism} if there exists $D\geq 0$ such that
$$
|\psi(g) - \psi(gh) + \psi(h)| \leq D,
$$
for every $g,h\in G$. The smallest number $D$ with the above
property is called the {\em defect} of $\psi$. If $\psi(g^n)=n\psi(g)$
for every $n\in \B Z$ and every $g\in G$ then $\psi$ is called
{\em homogeneous}; see \cite{MR2527432} for background on quasimorphisms.

A function $\nu\colon G\to \B R$ such that for all $g,h\in G$:
\begin{itemize}
\item $\nu(g)\geq 0$,
\item $\nu(g) = 0 $ if and only if $g=1$,
\item $\nu(gh)\leq \nu(g)+\nu(h)$
\end{itemize}
is called a {\em norm} on a group $G$. If in addition
$\nu(hgh^{-1})=\nu(g)$ then $\nu$ is called {\em conjugation invariant}.
The supremum $\nu(G)=\sup\{\nu(g)\ |\ g\in G\}$ is called the
{\em diameter} of $\nu$ or the diameter of $G$ with respect to $\nu$.
If $\nu(G)=\infty$ then $\nu$ is called {\em unbounded}.

The number $\tau(g) = \lim_{n\to \infty}\frac{\nu(g^n)}{n}$ is called
the {\em translation length} of $g$ with respect to the norm $\nu$.
If a group $G$ contains an element with positive translation length
with respect to the norm $\nu$ then $\nu$ is called {\em stably unbounded}.

\begin{example}\label{E:}
Let $\Sigma_{\infty}$ be an infinite symmetric group. That is, a group 
of finitely supported bijections of a countably infinite set. The cardinality
of the support defines a conjugation invariant norm of infinite diameter
in which every element has translation length equal to zero. 
\hfill $\diamondsuit$
\end{example}

\section{Proofs}\label{S:proofs}
Let $(M,d)$ be a complete Riemannian manifold. Let $I\colon M\to M$
be a geodesic symmetry through $x\in M$. By an abuse of notation we denote the
induced automorphism of the fundamental group by $I\colon \pi_1(M,x)\to
\pi_1(M,x)$.  Define the following two subsets of the fundamental group of $M$:
\begin{align*}
S &= \left\{ g\in \pi_1(M,x) \ |\ I(g) = g^{-1}\right\}\\
C &= \left\{ g\in \pi_1(M,x) \ |\ g = [c], \ c\in C_x\right \}.
\end{align*}
Recall that $C_x$ denotes the set of all closed local geodesics through 
$x\in~M$.

\begin{lemma}\label{L:S=C}
If every element of $\pi_1(M,x)$ has a unique geodesic representative then 
$S = C$.
\end{lemma}
\begin{proof}
If $c$ is a closed local geodesic through $x$, then $\bar{c}$ and $I(c)$ 
have a common initial segment. Thus, by the uniqueness of extension of
geodesics, $\bar{c} = I(c)$ and $I[c] = [c]^{-1}$.
This proves that $C\subseteq S$.

Let $\gamma$ be a local geodesic representing $s\in S$. Since 
$I(s) = s^{-1}$ is, on the one hand, represented by a geodesic
segment $I\circ \gamma$ and, on the other hand, by a local geodesic
$\overline{\gamma}$, we get that $I\circ \gamma = \overline{\gamma}$
due to the uniqueness of geodesic representatives. This implies that
$\gamma \in C_x$ and hence $s\in C$.
\end{proof}

\begin{remark}
Observe that it is important here that $M$ is a manifold. More precisely,
that a geodesic is uniquely determined by its initial segment. For example,
the graph presented on the figure below admits a geodesic
symmetry through $x$ but its closed geodesics going around one of the squares
are not preserved setwise.
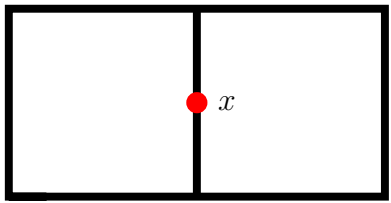
\begin{figure}[h]\label{F:}
\begin{tikzpicture}[line width=3pt, scale=0.5]
\draw (0,0) -- (10,0) -- (10,5) -- (0,5) -- (0,0) -- (1,0);
\draw (5,0) -- (5,5);
\filldraw[red]
(5,2.5) circle (5pt);
\draw (5.8,2.5) node {$x$};
\end{tikzpicture}
\caption{Warning example. }
\end{figure}
\end{remark}

\begin{lemma}\label{L:normal}
If every element of $\pi_1(M,x)$ has a unique local geodesic representative then 
the subgroup $\Gamma_x\subseteq \pi_1(M,x)$ is normal.
\end{lemma}
\begin{proof}
According to Lemma \ref{L:S=C} the subgroup $\Gamma_x$ is generated by
the set $S$. Let $s\in S$ and let $g\in \pi_1(M,x)$. Then
$$
gsg^{-1} = gsI_x(g^{-1}) \cdot I_x(g) g^{-1},
$$
which means that the conjugate of an element $s\in S$ is a product of
two elements from $S$.
\end{proof}

\begin{corollary}\label{C:domination}
If every element of $\pi_1(M,x)$ has a unique local geodesic representative then 
the closed geodesic norm on $\Gamma_x$ is dominated by a norm invariant
with respect to the conjugation action of $\pi_1(M,x)$.
\end{corollary}

\begin{proof}
Let $\overline{S} = \bigcup_{g\in \pi_1(M,x)} g S g^{-1}$. Since $\overline{S}$ is
invariant under conjugations by elements of $\pi_1(M,x)$ the associated
word norm is invariant under the conjugation action by $\pi_1(M,x)$.

Let $\|g\|_{\overline{S}}=n$. This means that $g=s_1^{g_{1}}\cdots s_n^{g_{n}}$,
where $s_i\in S$ and $g_i\in \pi_1(M,x)$.
It follows from Lemma \ref{L:normal} that
$$
\|g\|_S = \|s_1^{g_{1}}\cdots s_n^{g_{n}}\|_S \leq 2 n = 2\|g\|_{\overline{S}}.
$$
\end{proof}

\begin{proof}[Proof of Theorem \ref{T:chevalley}]
If $\Gamma_x$ is a finite group then there is nothing to prove. So assume that
$\Gamma_x$ is infinite. Since it is normal in $\pi_1(M,x)$, it is of finite 
index, according to \cite[(5.3) Proposition, p.324]{MR1090825}. 
If follows from \cite{MR2819193,MR2819193-add}
that every conjugation invariant norm on a finite index subgroup of an
S-arithmetic Chevalley group of rank at least $2$ has finite diameter
and hence the statement follows from Corollary \ref{C:domination}.
\end{proof}

\begin{proof}[Proof of Theorem \ref{T:negative}]
Let $\alpha\in \Omega^1(M)$ be a differential $1$-form and let
$\psi_{\alpha}\colon \pi_1(M,x)\to \B R$ be defined by
$$
\psi_{\alpha}(g) = \int_{\gamma} \alpha,
$$
where $\gamma$ is a closed local geodesic representing the conjugacy class of $g$.
The map $\psi_{\alpha}$ is a homogeneous quasi-morphism, 
see \cite[Example 2.3.1]{MR2527432}.

If $\alpha$ is such that $I^*(\alpha) = \alpha$ then $\psi_{\alpha}$ vanishes
on $S$. Indeed,
\begin{align*}
\psi_{\alpha}(s) &= \int_{c} \alpha = \int_{I\circ \overline{c}}\alpha\\
&= \int_{\overline{c}} I^*(\alpha) = -\int_{c} \alpha = -\psi_{\alpha}(s),
\end{align*}
which implies that $\psi_{\alpha}(s) = 0$.

Let $s_1,s_2\in S$ be two noncommuting elements. Then 
$$
I[s_1,s_2] = \left[ I(s_1),I(s_2) \right] = \left[ s_1^{-1},s_2^{-1} \right]
$$
is conjugate to $[s_1,s_2]$. Let $\gamma$ be a closed local geodesic
representing the conjugacy class of $[s_1,s_2]$. It follows that 
$I\circ \gamma = \gamma$ (up to reparametrisation by a shift).

Let $\beta$ be a $1$-form supported in a small
ball such that $\int_{\gamma}\beta >0$, where $\gamma$ is as above. 
Let $\alpha = \beta+I^*(\beta)$. Then
\begin{align*}
\psi_{\alpha}[s_1,s_2] &=
\int_{\gamma}\alpha = \int_{\gamma} \beta + I^*(\beta)\\
&= \int_{\gamma}\beta + \int_{I\circ \gamma}\beta
= 2\int_{\gamma}\beta >0.  
\end{align*}

We thus constructed a nontrivial homogeneous quasimorphism which vanishes on
the generating set $S$. A standard computation shows that it is Lipschitz
with respect to the word norm associated with $S$:
$$
|\psi_{\alpha}(g)| = |\psi_{\alpha}(s_1\cdots s_n)| \leq
\sum_{i=1}^n |\psi_{\alpha}(s_i)| + \|g\|_S D = D \|g\|_S.
$$
where $D$ is the defect of $\psi_{\alpha}$. We finally obtain that
$$
0<n|\psi_{\alpha}\left( [s_1,s_2] \right)| =
|\psi_{\alpha}([s_1,s_2]^n)| \leq D\|[s_1,s_2]^n\|_S
$$
which shows that the translation length of $[s_1,s_2]$ is positive and
hence the closed geodesic norm on $\Gamma_x$ is stably unbounded.
In particular, it has infinite diameter.
\end{proof}

\begin{corollary}\label{C:translation}
Let $(M,d)$ be as in Theorem \ref{T:negative}. Let $s_1,s_2\in S$ be
noncommuting elements represented by closed local geodesics.  Then 
their commutator $[s_1,s_2]$ has positive translation length with respect
to the closed geodesic norm. 
\qed
\end{corollary}

\bibliography{/home/kedra/sync/bibliography}

\def\polhk#1{\setbox0=\hbox{#1}{\ooalign{\hidewidth
  \lower1.5ex\hbox{`}\hidewidth\crcr\unhbox0}}}
  \def\polhk#1{\setbox0=\hbox{#1}{\ooalign{\hidewidth
  \lower1.5ex\hbox{`}\hidewidth\crcr\unhbox0}}}
  \def\polhk#1{\setbox0=\hbox{#1}{\ooalign{\hidewidth
  \lower1.5ex\hbox{`}\hidewidth\crcr\unhbox0}}}
  \def\polhk#1{\setbox0=\hbox{#1}{\ooalign{\hidewidth
  \lower1.5ex\hbox{`}\hidewidth\crcr\unhbox0}}}
  \def\polhk#1{\setbox0=\hbox{#1}{\ooalign{\hidewidth
  \lower1.5ex\hbox{`}\hidewidth\crcr\unhbox0}}}
  \def\polhk#1{\setbox0=\hbox{#1}{\ooalign{\hidewidth
  \lower1.5ex\hbox{`}\hidewidth\crcr\unhbox0}}} \def\cprime{$'$}
\begin{thebibliography}{10}

\bibitem{MR630586}
W.~Ballmann, G.~Thorbergsson, and W.~Ziller.
\newblock Closed geodesics and the fundamental group.
\newblock {\em Duke Math. J.}, 48(3):585--588, 1981.

\bibitem{MR2125453}
Valery Bardakov, Vladimir Shpilrain, and Vladimir Tolstykh.
\newblock On the palindromic and primitive widths of a free group.
\newblock {\em J. Algebra}, 285(2):574--585, 2005.

\bibitem{MR3730755}
Michael Brandenbursky, Jarek K\polhk~edra, and Egor Shelukhin.
\newblock On the autonomous norm on the group of {H}amiltonian diffeomorphisms
  of the torus.
\newblock {\em Commun. Contemp. Math.}, 20(2):1750042, 27, 2018.

\bibitem{MR1744486}
Martin~R. Bridson and Andr{\'e} Haefliger.
\newblock {\em Metric spaces of non-positive curvature}, volume 319 of {\em
  Grundlehren der Mathematischen Wissenschaften [Fundamental Principles of
  Mathematical Sciences]}.
\newblock Springer-Verlag, Berlin, 1999.

\bibitem{MR2527432}
Danny Calegari.
\newblock {\em scl}, volume~20 of {\em MSJ Memoirs}.
\newblock Mathematical Society of Japan, Tokyo, 2009.

\bibitem{MR1243786}
Ted Chinburg and Alan~W. Reid.
\newblock Closed hyperbolic {$3$}-manifolds whose closed geodesics all are
  simple.
\newblock {\em J. Differential Geom.}, 38(3):545--558, 1993.

\bibitem{MR2819193}
{\'S}wiatos{\l}aw~R. Gal and Jarek K\k{e}dra.
\newblock On bi-invariant word metrics.
\newblock {\em J. Topol. Anal.}, 3(2):161--175, 2011.

\bibitem{MR2819193-add}
{\'S}wiatos{\l}aw~R. Gal and Jarek K\k{e}dra.
\newblock Finite index subgroups in chevalley groups are bounded: an addendum
  to "on bi-invariant word metrics".
\newblock {\em arXiv:1808.06376}, 2018.

\bibitem{MR1458971}
Kerry~N. Jones and Alan~W. Reid.
\newblock Geodesic intersections in arithmetic hyperbolic {$3$}-manifolds.
\newblock {\em Duke Math. J.}, 89(1):75--86, 1997.

\bibitem{MR1330918}
Wilhelm P.~A. Klingenberg.
\newblock {\em Riemannian geometry}, volume~1 of {\em De Gruyter Studies in
  Mathematics}.
\newblock Walter de Gruyter \& Co., Berlin, second edition, 1995.

\bibitem{MR1090825}
G.~A. Margulis.
\newblock {\em Discrete subgroups of semisimple {L}ie groups}, volume~17 of
  {\em Ergebnisse der Mathematik und ihrer Grenzgebiete (3) [Results in
  Mathematics and Related Areas (3)]}.
\newblock Springer-Verlag, Berlin, 1991.

\bibitem{MR0232311}
J.~Milnor.
\newblock A note on curvature and fundamental group.
\newblock {\em J. Differential Geometry}, 2:1--7, 1968.

\bibitem{MR0075634}
A.~S. Schwarz.
\newblock A volume invariant of coverings.
\newblock {\em Dokl. Akad. Nauk SSSR (N.S.)}, 105:32--34, 1955.

\end{thebibliography}
\bibliographystyle{plain}

\end{document}